\documentclass[smallextended]{svjour3}

\smartqed  

\usepackage{amsmath,amsfonts,amssymb,enumerate}

\newcommand{\N}{\mathbb{N}}                   
\newcommand{\R}{\mathbb{R}}                   
\newcommand{\C}{\mathbb{C}}                   
\newcommand{\K}{\mathbb{K}}                   
\newcommand{\bU}{\overline{U}}
\newcommand{\bV}{\overline{V}}
\newcommand{\bS}{\overline{S}}
\newcommand{\tSigma}{\widetilde{\Sigma}}
\newcommand{\fbd}{\mathrm{fbd}}               
\newcommand{\vp}{\varphi}                     
\newcommand{\OXxi}{{\mathcal O}_{X,\xi}}      
\newcommand{\Cx}{\C\!\left<x\right>}          
\newcommand{\dl}{\mathfrak{d}}                
\newcommand{\Dl}{\mathfrak{D}}                
\newcommand{\OO}{\mathcal{O}}                 

\numberwithin{equation}{section}

\journalname{Mathematische Annalen}
\begin{document}

\title{Tameness of holomorphic closure dimension\\ in a semialgebraic set
 \thanks{J.Adamus's research was partially supported by Natural Sciences and Engineering Research Council of Canada.}}
\titlerunning{Tameness of holomorphic closure dimension in a semialgebraic set}

\author{Janusz Adamus         \and
        Serge Randriambololona
}

\institute{J. Adamus \at
              Department of Mathematics, The University of Western Ontario, London, Ontario N6A 5B7 Canada - and -
              Institute of Mathematics, Faculty of Mathematics and Computer Science,
              Jagiellonian University, ul. {\L}ojasiewicza 6, 30-348 Krak{\'o}w, Poland \\
              Tel.: +1-519-661-3638\\
              Fax: +1-519-661-3610\\
              \email{jadamus@uwo.ca}
           \and
           S. Randriambololona \at
              Department of Mathematics, The University of Western Ontario, London, Ontario N6A 5B7 Canada
}

\date{Received: date / Accepted: date}

\maketitle

\begin{abstract}
Given a semianalytic set $S$ in $\C^n$ and a point $p\in\bS$, there is a unique smallest complex-analytic germ $X_p$ which contains $S_p$, called the holomorphic closure of $S_p$. We show that if $S$ is semialgebraic then $X_p$ is a Nash germ, for every $p$, and $S$ admits a semialgebraic filtration by the holomorphic closure dimension. As a consequence, every semialgebraic subset of a complex vector space admits a semialgebraic stratification into CR manifolds.
\end{abstract}


\section{Introduction}
\label{sec:intro}

Given a real-analytic (or, more generally, semianalytic) subset $S$ of an open set in a complex vector space, a natural question arises how much of the ambient complex structure is inherited (locally) by $S$. In the present paper, we are interested in the following local biholomorphic invariant of $S$: Let $\xi $ be a point of $S$.
We shall consider the minimal dimension of a complex-analytic germ at $\xi$ containing the germ $S_\xi$ - the so-called \emph{holomorphic closure dimension} of $S_\xi$, denoted $\dim_{HC}S_\xi$ (see~\cite{AS}). The minimal (with respect to inclusion) complex-analytic germ $X_\xi$ containing $S_\xi$ is called the \emph{holomorphic closure} of $S_\xi$. For $d\in\N$, let $\mathcal{S}^d(S)$ denote the set of points $\xi\in S$ for which $\dim_{HC}S_\xi\geq d$. We will investigate the structure of the sets $\mathcal{S}^d(S)$, and how they are modified by holomorphic mappings of the ambient space.

The study of this outer complex dimension of real-analytic sets originates in \cite{S}. It is well motivated in classical CR geometry. The holomorphic closure dimension of a real-analytic CR manifold is complementary to its CR dimension, and its properties can be used to show that a real-analytic manifold is a CR manifold outside a nowhere-dense semianalytic subset (Proposition~1.4 and Theorem~1.5 of \cite{AS}). However, when considering the sets $\mathcal{S}^d(S)$, there is no gain in assuming that $S$ is non-singular, and so we shall consider the general, singular case.

On the other hand, our personal bias is to study the real-analytic and semianalytic sets for themselves. It seems natural to expect the sets $\mathcal{S}^d(S)$ to remain close to the class of $S$. By comparison, in \cite{ARS}, we studied the sets $\mathcal{A}^d(S)$ of those $\xi\in S$ for which $S_\xi$ contains a complex analytic germ of dimension at least $d$. Theorem~1.1 of \cite{ARS} asserts that the $\mathcal{A}^d(S)$ are semianalytic (not necessarily real-analytic though). Alas, things do not look so good for the holomorphic closure dimension. In fact, Example~6.3 of \cite{AS} shows that, for $d$ greater than the generic holomorphic closure dimension of $S$, the sets $\mathcal{S}^d(S)$ need not even be subanalytic!
Hence, in order to hope for some sort of tameness, the next largest class to consider is that of semialgebraic sets, which is the class we consider here.
For that reason, when considering how a holomorphic mapping modifies the complex structure inherited by our set $S$, we also need to restrict to mappings which preserve the class of $S$. Thus our general objects of study are semialgebraic subsets of complex vector spaces and holomorphic semialgebraic mappings between (open subsets of) such spaces. As it turns out, all such maps are Nash (Proposition~\ref{prop:nash-map}).
\medskip

The first part of this article is concerned with tameness of the holomorphic closure dimension in semialgebraic subsets of complex vector spaces.
In Section~\ref{sec:nash-image}, we study the images of Nash sets by Nash mappings, and prove a local variant of Chevalley's theorem on images of algebraic mappings (Theorem~\ref{thm:nash-map-image}). Although the image of a Nash set by a Nash mapping need not be Nash (not even Nash-constructible, Remark~\ref{rem:no-local-strong-Chevalley}), its holomorphic closure is already Nash, and of the same dimension as the image itself.

The key consequence of Theorem~\ref{thm:nash-map-image} is the following result, which lies at the heart of all our tameness arguments.

\begin{proposition}
\label{prop:HC-nash}
The holomorphic closure of a semialgebraic set $S$ at a point $\xi\in\bS$ is a Nash germ.
\end{proposition}

We prove this proposition in Section~\ref{sec:HC-stratification}. The main result of Section~\ref{sec:HC-stratification} is the proof of the semialgebraic stratification by holomorphic closure dimension:

\begin{theorem}
\label{thm:HC-stratification}
Let $S$ be a semialgebraic subset of a finite-dimensional complex vector space $M$. Then the sets $\mathcal{S}^d(S)$ are semialgebraic and closed in $S$, for all $d\in\N$.
\end{theorem}

\begin{remark}
\label{rem:no-analytic-HC-tameness}
An analogous semianalytic filtration does not exist in general for $S$ semianalytic. Indeed, Example~6.3 of \cite{AS} shows a connected non-singular $\R$-analytic set $R\subset\C^5$ of generic holomorphic closure dimension 3, and with the set $\mathcal{S}^4(R)$ non-empty and not semianalytic. In fact, $\mathcal{S}^4(R)$ is not even subanalytic.
\end{remark}

Section~\ref{sec:preimages} gives a series of results concerning the relationship between the holomorphic closure dimension of a semialgebraic set and that of its preimage under a holomorphic semialgebraic mapping. In the last section, we present applications of these results to CR geometry. Our main application is the following semialgebraic stratification by CR manifolds:

\begin{theorem}
\label{thm:semialg-CR-stratification}
Let $S$ be a semialgebraic subset of a finite-dimensional complex vector space $M$. Then there exists a finite partition $\{S_\iota\}_{\iota\in I}$ of $S$ into semialgebraic subsets of $M$ satisfying the following conditions:
\begin{itemize}
\item[(i)] Every $S_\iota$ is a CR manifold.
\item[(ii)] Compatibility with the family $\{\mathcal{S}^d(S)\}_{d\in\N}$: for each $\iota\in I$ there exists $d\in\N$ such that $S_\iota\subset\mathcal{S}^d(S)\setminus\mathcal{S}^{d+1}(S)$.
\item[(iii)] Condition of the frontier: For any $j,k\in I$, either $S_j\cap\overline{S_k}=\varnothing$ or else $S_j\subset\overline{S_k}$, $\dim_\R S_j<\dim_\R S_k$ and $\dim_{HC}(S_j)_\xi\leq\dim_{HC}(S_k)_\xi$ for every $\xi\in S_j$.
\end{itemize}
\end{theorem}

This result, again, is in contrast with the real-analytic case (see Remark~\ref{rem:no-analytic-CR-tameness}).
Finally, we investigate how the CR structure varies under some holomorphic mappings. Particularly, we study holomorphic semialgebraic desingularizations.

To make the article easily accessible for both real-algebraic and CR geometry communities, we review the basic notions and tools in the next section.


\section{Preliminaries}
\label{sec:preliminaries}

Throughout this article, the dimension of a set $X$ is understood in the following sense. If $X$ is a subset of a $\K$-manifold $M$ ($\K=\R$ or $\C$), then
\[
\dim_{\K} X =\max\{\dim_{\K}N :\ N\subset X	,\ N\mathrm{\ a\ closed\ submanifold\ of\ an\ open\ subset\ of\ }M\}\,.
\]
The dimension of the germ $X_p$ of a set $X$ at a point $p\in M$ is then defined as
\[
\dim_{\K} X_p = \min\{\dim_{\K}(X\cap U) :\ U\mathrm{\ an\ open\ neighbourhood\ of\ } p \mathrm{\ in\ }M\}\,.
\]

Because real- and complex-analytic and algebraic objects are often considered next to one another in our arguments, we will always call them $\R$-analytic (resp. $\R$-algebraic) or $\C$-analytic (resp. $\C$-algebraic), to avoid confusion. When speaking of dimension of a $\K$-analytic (or $\K$-algebraic) set, we always mean its $\K$-dimension in the above sense, unless otherwise specified.

\subsection{Semialgebraic sets}
\label{subsec:semialg}

For a concise introduction to semialgebraic geometry, we refer the reader to \cite[Ch.\,2]{BCR} and \cite{Cos}. 

Let $M$ be a finite-dimensional $\R$-vector space. A choice of base for $M$ gives a linear isomorphism $\psi:\R^n\to M$, where $n=\dim{M}$. We say that a function $f:M\to\R$ is a \emph{polynomial function} on $M$ if there exists $P\in\R[X_1,\dots,X_n]$ such that $(f\circ\psi)(x)=P(x)$ for all $x=(x_1,\dots,x_n)\in\R^n$. Since linear base change is a polynomial mapping (with polynomial inverse), it follows that the above definition is independent of the choice of base for $M$. We say that a subset $S$ of $M$ is \emph{semialgebraic} if $S$ is a finite union of sets of the form
\[
\{x\in M: f_1(x)=\dots=f_r(x)=0,g_1(x)>0,\dots,g_s(x)>0 \}\,,
\]
where $r,s\in\N$ and $f_1,\dots,f_r,g_1,\dots,g_s$ are polynomial functions on $M$.
One easily checks that the union and intersection of two semialgebraic sets are semialgebraic, as is the complement of a semialgebraic set.

Let $\Omega$ and $\Delta$ be open subsets of finite-dimensional $\R$-vector spaces $M$ and $N$ respectively. A mapping $\vp:\Omega\to\Delta$ is called a \emph{semialgebraic mapping} if its graph is a semialgebraic subset of $M\times N$.
The Tarski-Seidenberg Theorem (see, e.g., \cite[Prop.\,2.2.7]{BCR}) insures that the image (resp. the inverse image) by $\vp$ of a semialgebraic subset of $M$ (resp. $N$) is semialgebraic in $N$ (resp. $M$).

\begin{remark}
\label{rem:SAG-facts}
{~}
\begin{enumerate}
\label{rem:SA-facts}

\item \cite[Prop.\,2.2.2]{BCR}. If $S$ is semialgebraic in $M$, then the topological closure and interior of $S$ in $M$ are semialgebraic sets.

\item \cite[Thm.\,2.4.4]{BCR}. Every semialgebraic subset of $M$ has a finite number of connected components, each of which is semialgebraic in $M$.

\item \cite[Thm.\,2.9.10]{BCR}. Every semialgebraic subset of $M$ is a disjoint union of a finite family of sets, 
each of which is a connected $\R$-analytic manifold and a semialgebraic subset of $M$.
\end{enumerate}
\end{remark}

A semialgebraic subset of a $\C$-vector space $M$ is one that is semialgebraic in $M$ regarded as an $\R$-vector space. If $\Omega$ and $\Delta$ are open subsets of finite-dimensional $\C$-vector spaces $M$ and $N$ respectively, then a \emph{holomorphic semialgebraic} mapping $\vp:\Omega\to\Delta$ is a holomorphic map whose graph is a semialgebraic subset of $M\times N$ in the above sense.

\subsection{Nash sets and Nash mappings}
\label{subsec:nash}

See \cite{Tw} for a detailed exposition of (complex) Nash sets and Nash mappings.

Let $M$ be a finite-dimensional $\C$-vector space. Let $\Omega$ be an open subset of $M$, and let $f$ be a holomorphic function on $\Omega$. 
We say that $f$ is a \emph{Nash function} at $x_0\in\Omega$ if there exist an open neighbourhood $U$ of $x_0$ in $\Omega$ and 
a $\C$-polynomial function $P:M\times\C\to\C$, $P\neq0$, such that $P(x,f(x))=0$ for $x\in U$. 
A holomorphic function on $\Omega$ is a Nash function if it is a Nash function at every point of $\Omega$. 
Let $N$ be another finite-dimensional $\C$-vector space. A holomorphic mapping $\vp:\Omega\to N$ is a \emph{Nash mapping} 
if each of its components is a Nash function on $\Omega$ with respect to some basis of $N$.

A subset $X$ of $\Omega$ is called a \emph{Nash subset} of $\Omega$ if for every $x_0\in\Omega$ there exist an open neighbourhood $U$ of $x_0$ 
in $\Omega$ and Nash functions $f_1,\dots,f_s$ on $U$, such that $X\cap U=\{x\in U: f_1(x)=\dots=f_s(x)=0\}$. 
A germ $X_\xi$ at $\xi\in\Omega$ is a \emph{Nash germ} if there exists an open neighbourhood $U$ of $\xi$ in $\Omega$ such that $X\cap U$ is a Nash subset of $U$.

\begin{remark}
\label{rem:nash-germ}
Equivalently, $X_\xi$ is a Nash germ if its defining ideal can be generated by power series algebraic over the polynomial ring $\C[x]$; that is, $\OXxi\cong\C\{x\}/(f_1,\dots,f_s)\C\{x\}$ with $f_j\in\Cx$, $j=1,\dots,s$, where $x=(x_1,\dots,x_m)$ and $\Cx$ denotes the algebraic closure of $\C[x]_{(x)}$ in $\C[[x]]$.
\end{remark}

\begin{remark}
\label{rem:nash-facts}
{~}
\begin{enumerate}
\item \cite[Thm.\,2.15]{Tw}. A holomorphic $\vp:\Omega\to N$ (resp. germ $\vp_\xi$ of $\vp$ at $\xi\in\Omega$) is a Nash mapping (resp. Nash map-germ) if and only if its graph is a Nash subset of $\Omega\times N$ (resp. a Nash germ at $(\xi,f(\xi))\in\Omega\times N$).

\item \cite[Prop.\,2.6]{Tw}. If $\{X_\iota\}_{\iota\in I}$ is a family of Nash subsets of $\Omega$, then $\bigcap_{\iota\in I} X_\iota$ is a Nash subset of $\Omega$. If moreover the family $\{X_\iota\}_{\iota\in I}$ is locally finite, then $\bigcup_{\iota\in I} X_\iota$ is also Nash in $\Omega$.

\item \cite[Thm.\,2.10]{Tw}. Let $X$ be an irreducible $\C$-analytic subset of an open set $\Omega\subset M$. Then $X$ is a Nash subset of $\Omega$ if and only if there exists an irreducible algebraic subset $Z$ of $M$ such that $X$ is an analytic-irreducible component of $\Omega\cap Z$.

\item \cite[Thm.\,2.11]{Tw}. Let $X$ be a Nash subset of an open set $\Omega\subset M$, and let $Y$ be an irreducible component of $X$. Then $Y$ is a Nash subset of $\Omega$.

\item \cite[Thm.\,2.12]{Tw}. An irreducible Nash subset of the space $M$ is an irreducible algebraic subset of $M$.
\end{enumerate}
\end{remark}

\subsection{CR structure}
\label{subsec:CR}

There are many excellent monographs on CR geometry; see, e.g., \cite{BER}, \cite{Bog}, or \cite{DA2}.

Given an $\R$-linear subspace $L$ in $\C^n$ of dimension $d$, one defines the {\it CR dimension} of $L$ to be the largest $m$ 
such that $L$ contains a $\C$-linear subspace of (complex) dimension $m$. Clearly, $0\le m\le \left[\frac{d}{2}\right]$. 
A real submanifold $M$ in $\C^n$ of real dimension $d$ is called a {\it CR manifold} of CR dimension $m$, if the tangent 
space $T_p M$ has CR dimension $m$ for every point $p\in M$. We write 
$\dim_{CR}M=m$. In particular, if $m=0$, then $M$ is called a {\it totally real} submanifold.

\subsection{Real-analytic subgerms of complex-analytic germs and complexification}
\label{subsec:complexification}

We recall the following construction from \cite{AS}.
Let $\dl:\C^n_{\zeta}\to\C^{2n}_{(z,w)}$ be the map defined by $\dl(\zeta)=(\zeta,\bar\zeta)$. Then 
$\Dl=\dl(\C^n)$ is a totally real embedding of $\C^n$ into $\C^{2n}$. Suppose $R$ is an $\R$-analytic set in 
$\C^n$, and $p\in R$. With a moderate abuse of notation, we will denote by $R_p^c$ the complexification of the 
germ of $\dl(R)$ at $q:=\dl(p)$ in $\C^{2n}$, that is, the smallest germ of a $\C$-analytic set in $\C^{2n}$ 
which contains the germ of $\dl(R)$ at~$q$. (And we will call it the complexification of $R_p$, for short.)

Let now $X$ be a $\C$-analytic set in an open neighbourhood $U$ of $p$ in $\C^n$, defined by 
$g_1,\dots,g_t\in\OO(U)$, where $g_k(\zeta)=\sum_{|\nu|\geq0}c_{k\nu}\zeta^\nu$, $k=1,\dots,t$. We set
\begin{equation}
\label{eq:xz}
\begin{split}
 X^z &=\{(z,w)\in U': g_k(z)=\sum_{|\nu|\ge 0} c_{k\nu}z^{\nu}=0,\; k=1,\dots,t\}\\
 X^w &=\{(z,w)\in U': \bar{g}_k(w)=\sum_{|\nu|\ge 0} \overline{c}_{k\nu} w^{\nu}=0,\; k=1,\dots,t\}\,,
\end{split}
\end{equation}
where $U'$ is some small open neighbourhood of $q$ in $\C^{2n}$. Let $\pi^z:\C^{2n}_{(z,w)}\to\C^n_z$ 
and $\pi^w:\C^{2n}_{(z,w)}\to\C^n_w$ be the coordinate projections. Then $X^z=\pi^z(X^z)\times\C^n$ and 
$X^w=\C^n\times\pi^w(X^w)$, as the defining equations of $X^z$ (resp. $X^w$) do not involve variables $w$ 
(resp. $z$). Therefore, the set
\begin{equation}
\label{eq:x-hat}
\widehat X:= X^z\cap X^w=\pi^z(X^z)\times\pi^w(X^w)
\end{equation}
is $\C$-analytic (in $U'$), of dimension equal twice the dimension of $X$. If $X_p$ is irreducible, then the 
complexification $X^c_p$ of $X_p$ (viewed as an $\R$-analytic germ) coincides with $\widehat{X}_q$.
Indeed, clearly $\dl(X)\subset\widehat{X}$, hence $X^c_p\subset\widehat{X}_q$. But the irreducibility of $X_p$ implies that of $\widehat{X}_q$ (by~\eqref{eq:xz} and~\eqref{eq:x-hat}), and $\dim X^c_p=2\dim X_p=\dim\widehat{X}_q$, so $X^c_p=\widehat{X}_q$.

Let $A\subset\C^{2n}$ be a $\C$-analytic representative of the complexification $R^c_p$ at $q$ (in some open neighbourhood of $q$); i.e., $A_q=R^c_p$.
Then the holomorphic closure $\overline{R_p}^{HC}$ of the germ $R_p$ can be identified with the smallest $\C$-analytic germ containing $(\pi^z(A))_{\pi^z(q)}$.
Indeed, on the one hand we have
\[
R_p\subset X_p\ \Rightarrow\ A_q\subset \widehat{X}_q\ \Rightarrow\ (\pi^z(A))_{\pi^z(q)}\subset(\pi^z(X^z))_{\pi^z(q)}.
\]
On the other hand, suppose $(\pi^z(A))_{\pi^z(q)}\subset\tilde{Z}_{\pi^z(q)}$ for some $\C$-analytic $\tilde{Z}$ 
in a neighbourhood $V$ of $\pi^z(q)$ in $\C^n$. Say, $\tilde{Z}=\{z\in V: g_k(z)=0, k=1,\dots,t\}$.
Define a $\C$-analytic set $Z$ in a neighbourhood $U$ of $p$ as $Z=\{\zeta\in U: g_k(\zeta)=0, k=1,\dots,t\}$.
Then $Z=\dl^{-1}((\pi^z)^{-1}(\tilde{Z})\cap\Dl)$, and hence
\[
R_p\subset(\dl^{-1}(A\cap\Dl))_p\subset(\dl^{-1}((\pi^z)^{-1}(\tilde{Z})\cap\Dl))_p=Z_p\,.
\]

\section{Images of Nash sets under Nash mappings}
\label{sec:nash-image}

Our main tools in this section are Remmert's Rank Theorem and Chevalley's theorem on constructibility of images of algebraic sets. We recall them below for reader's convenience, in the form most suitable for our purposes (see, e.g., \cite{Loj}). We will denote by $\fbd_x\vp$ the dimension at $x$ of a fibre $\vp^{-1}(\vp(x))$.

\begin{theorem}[Remmert]
\label{thm:rem}
If $\vp:X\to Y$ is a holomorphic mapping of $\C$-analytic sets such that $\fbd_x\vp=k$ for all $x\in X$, then every point $\xi\in X$ admits an arbitrarily small open neighbourhood $W$ in $X$, such that $\vp(W)$ is locally analytic in $Y$, of dimension $\dim_\xi X-k$.
\end{theorem}

\begin{theorem}[Chevalley]
\label{thm:chev}
Suppose that $M$ and $N$ are finite-dimensional $\C$-vector spaces, and let $\pi:M\times N\to N$ denote the canonical projection.
If $Z$ is an irreducible $\C$-algebraic subset of $M\times N$, then $\overline{\pi(Z)}$ is an irreducible $\C$-algebraic subset of $N$, and $\dim\overline{\pi(Z)}=\dim\pi(Z)$.
(Here $\overline{\pi(Z)}$ is the closure of $\pi(Z)$ in the Euclidean topology of $N$.)
\end{theorem}
\medskip

From now on, $M$ and $N$ denote finite-dimensional $\C$-vector spaces, and $\pi:M\times N\to N$ is the canonical projection. For a subset $W$ of an open set $V$ in $N$, we will denote by $\overline{W}^{HC(V)}$ the smallest $\C$-analytic subset of $V$ containing $W$ (by analogy to the holomorphic closure of a germ). We will need the following adaptation of Chevalley's theorem~\ref{thm:chev} to the local setting.

\begin{proposition}
\label{prop:nash-proj}
Let $Z$ be an irreducible $\C$-algebraic subset of $M\times N$, let $\Omega$ be a non-empty open set in $M\times N$, and let $X$ be an analytic-irreducible component of $Z\cap\Omega$. Then, for every point $(\xi,\eta)$ of $X$ and every pair of bounded open neighbourhoods $U$ of $\xi$ in $M$ and $V$ of $\eta$ in $N$ such that $\bU\times\bV\subset \Omega$, the set $\overline{\pi(X\cap(U\times V))}^{HC(V)}$ is a union of some analytic-irreducible components of $\overline{\pi(Z)}\cap V$. Moreover, $\dim\overline{\pi(X\cap(U\times V))}^{HC(V)}=\dim\pi(X\cap(U\times V))$.
\end{proposition}

\begin{proof}
Let $\lambda$ denote the minimal fibre dimension of the restriction $\pi|_Z:Z\to N$.
Set $\tSigma:=\{z\in Z:\fbd_z(\pi|_Z)>\lambda\}$. By Chevalley's theorem on upper semicontinuity of fibre dimension~\cite[Thm.\,13.1.3]{Gro}, $\tSigma$ is algebraic. Since $\tSigma$ is a proper subset of $Z$, and $Z$ is irreducible, it follows that $\dim\tSigma<\dim Z$. By Theorem~\ref{thm:chev}, $\overline{\pi(\tSigma)}$ is an algebraic subset of $N$, of dimension $\dim\pi(\tSigma)$. Note that
\begin{equation}
\label{eq:1}
\dim\overline{\pi(\tSigma)}\leq\dim\pi(Z)-2\,.
\end{equation}
Indeed, every irreducible component of $\tSigma$ is of dimension at most $\dim{Z}-1$, and the generic fibre dimension of $\pi$ restricted to such a component is at least $\lambda+1$. Hence
\[
\dim\overline{\pi(\tSigma)}=\dim\pi(\tSigma)\leq(\dim{Z}-1)-(\lambda+1)=(\dim{Z}-\lambda)-2=\dim\pi(Z)-2\,,
\]
where the inequality follows from~\cite[\S\,V.3.2, Thm.\,2]{Loj}.

Set $\Sigma:=\pi^{-1}(\overline{\pi(\tSigma)})\cap Z$. Then $\Sigma$ is an algebraic subset of $M\times N$ and $\Sigma\subset Z$. We claim that $\Sigma$ is a proper subset of $Z$ (or, equivalently, that $\dim\Sigma<\dim Z$, by irreducibility of $Z$).
Indeed, by surjectivity of $\pi$, we have $\pi(\pi^{-1}(\overline{\pi(\tSigma)}))=\overline{\pi(\tSigma)}$, hence
\[
\pi(\Sigma)=\pi(\pi^{-1}(\overline{\pi(\tSigma)})\cap Z)\subset\overline{\pi(\tSigma)}\,.
\]
Therefore $\dim\pi(\Sigma)\leq\dim\overline{\pi(\tSigma)}\leq\dim\pi(Z)-2$, by \eqref{eq:1}, and hence $\dim\Sigma<\dim{Z}$.
Since, by assumption, $X$ is analytic-irreducible and of the same dimension as $Z$, it follows that $X\cap\Sigma$ is nowhere-dense in $X$.

Let now $(\xi,\eta)$ be an arbitrary point of $X$, and let $U\times V$ be a relatively compact product neighbourhood of $(\xi,\eta)$ in $M\times N$ such that $\bU\times\bV\subset\Omega$.
By nowhere-density of $X\cap\Sigma$ in $X$, we have
\begin{equation}
\notag
\overline{\pi(X\cap(U\times V))}=\overline{\pi((X\setminus\Sigma)\cap(U\times V))}=\overline{\pi(X\cap(U\times V))\setminus\overline{\pi(\tSigma)}}\,,
\end{equation}
where the rightmost equality follows from the definition of $\Sigma$. In particular,
\begin{equation}
\label{eq:2}
\overline{\pi(X\cap(U\times V))}^{HC(V)}=\overline{\pi(X\cap(U\times V))\setminus\overline{\pi(\tSigma)}}^{HC(V)}\,,
\end{equation}
since every $\C$-analytic subset of $V$ is closed in $V$.

For every point $(x,y)\in (X\setminus\Sigma)\cap(U\times V)$, the projection $\pi|_X$ has constant fibre dimension $\lambda$ near $(x,y)$. (Indeed, it is clear for every $(x,y)$ such that $X_{(x,y)}=Z_{(x,y)}$, and for the other points it follows from upper semicontinuity of fibre dimension and the fact that $X\setminus\overline{Z\setminus X}$ is dense in $X$.) Hence, by Theorem~\ref{thm:rem}, 
there exist open neighbourhoods $U^{(x,y)}$ of $x$ in $U$ and $V^{(x,y)}$ of $y$ in $V$, such that $\pi(X\cap(U^{(x,y)}\times V^{(x,y)}))$ is analytic in an open neighbourhood of $y$ in $N$, of pure dimension $\dim{X}-\lambda=\dim{Z}-\lambda$. On the other hand, $\pi(X\cap(U^{(x,y)}\times V^{(x,y)}))$ is contained in the intersection of the algebraic set $\overline{\pi(Z)}$ with $V^{(x,y)}$, which is also of dimension $\dim{Z}-\lambda$. It follows that the germ $(\pi(X\cap(U^{(x,y)}\times V^{(x,y)})))_y$ is a union of some analytic-irreducible components of $(\overline{\pi(Z)})_y$.

For a fixed $y_0\in\pi(X\cap(U\times V))\setminus\overline{\pi(\tSigma)}$, consider the family of open sets $U^{(x,y_0)}\times V^{(x,y_0)}$ as above, over all $x\in U$ such that $(x,y_0)\in X$. By relative compactness of $U$, we can choose finitely many $x_1,\dots,x_s\in U$, such that
\begin{multline}
\notag
(\pi(X\cap(U\times V)))_{y_0}\\
=(\pi(X\cap(U^{(x_1,y_0)}\times V^{(x_1,y_0)})))_{y_0}\cup\dots\cup(\pi(X\cap(U^{(x_s,y_0)}\times V^{(x_s,y_0)})))_{y_0}\,,
\end{multline}
and hence the germ $(\pi(X\cap(U\times V)))_{y_0}$ is a union of some analytic-irreducible components of the germ $(\overline{\pi(Z)})_{y_0}$. Therefore $\pi(X\cap(U\times V))\setminus\overline{\pi(\tSigma)}$ is a locally analytic subset of $V\setminus\overline{\pi(\tSigma)}$, of pure dimension $\dim{Z}-\lambda$.

Next, note that, for every analytic-irreducible component $\Lambda$ of $\overline{\pi(Z)}\cap V$, if $(\pi(X\cap(U\times V)))_y\supset\Lambda_y$ at some point $y\in V\setminus\overline{\pi(\tSigma)}$, then $\overline{\pi(X\cap(U\times V))}^{HC(V)}\supset\Lambda$ (by irreducibility of $\Lambda$). Let $\{\Lambda_\iota\}_{\iota\in I}$ be the family of all the analytic-irreducible components of $\overline{\pi(Z)}\cap V$ whose germ at some point of $V\setminus\overline{\pi(\tSigma)}$ is contained in the corresponding germ of $\pi(X\cap(U\times V))$. Then
\begin{equation}
\label{eq:2a}
\bigcup_{\iota\in I}\Lambda_\iota\subset\overline{\pi(X\cap(U\times V))}^{HC(V)}\ \mathrm{and}\ \ 
\pi(X\cap(U\times V))\setminus\overline{\pi(\tSigma)}\subset\bigcup_{\iota\in I}(\Lambda_\iota\setminus\overline{\pi(\tSigma)})\,.
\end{equation}
On the other hand, by local finiteness of the family $\{\Lambda_\iota\}_{\iota\in I}$, we have
\[
\overline{\bigcup_{\iota\in I}(\Lambda_\iota\setminus\overline{\pi(\tSigma)})}=\bigcup_{\iota\in I}\overline{(\Lambda_\iota\setminus\overline{\pi(\tSigma)})}=\bigcup_{\iota\in I}\Lambda_\iota\,.
\]
The latter is $\C$-analytic in $V$, hence combining \eqref{eq:2} and \eqref{eq:2a}, we get
\[
\bigcup_{\iota\in I}\Lambda_\iota\subset\overline{\pi(X\cap(U\times V))}^{HC(V)}=\overline{\pi(X\cap(U\times V))\setminus\overline{\pi(\tSigma)}}^{HC(V)}
\subset\overline{\bigcup_{\iota\in I}\Lambda_\iota}^{HC(V)}=\bigcup_{\iota\in I}\Lambda_\iota\,;
\]
i.e., $\displaystyle{\overline{\pi(X\cap(U\times V))}^{HC(V)}=\bigcup_{\iota\in I}\Lambda_\iota}$.

The last assertion of the proposition follows from the fact that $\pi(X\cap(U\times V))\setminus\overline{\pi(\tSigma)}$ is of pure dimension $\dim\overline{\pi(Z)}$.
\qed
\end{proof}

\begin{corollary}
\label{cor:nash-proj}
Let $X$ be a Nash subset of an open set $\Omega\subset M\times N$. Let $U$ and $V$ be bounded open subsets of $M$ and $N$ respectively, such that $\bU\times\bV\subset\Omega$. Then $\overline{\pi(X\cap(U\times V))}^{HC(V)}$ is a Nash subset of $V$, of dimension equal to $\dim\pi(X\cap(U\times V))$.
\end{corollary}

\begin{proof}
Let $X=\bigcup_{\iota\in I}X_\iota$ be the decomposition of $X$ into (a locally finite family of) analytic-irreducible components. Then the family $\{X_\iota\cap(U\times V)\}_{\iota\in I}$ is finite, by compactness of $\bU\times \bV$. By Remark~\ref{rem:nash-facts}(4), each $X_\iota$ is a Nash subset of $\Omega$. Further, by irreducibility of $X_\iota$ and Remark~\ref{rem:nash-facts}(3), there exists, for every $\iota\in I$, an irreducible algebraic set $Z_\iota$ in $M\times N$ such that $X\iota$ is an analytic-irreducible component of $Z_\iota\cap\Omega$. Therefore $\overline{\pi(X_\iota\cap(U\times V))}^{HC(V)}$ is a Nash subset of $V$, by Proposition~\ref{prop:nash-proj}, and hence so is $\overline{\pi(X\cap(U\times V))}^{HC(V)}$, as
\[
\overline{\pi(X\cap(U\times V))}^{HC(V)}=\overline{\bigcup_{\iota\in I}\pi(X_\iota\cap(U\times V))}^{HC(V)}=\bigcup_{\iota\in I}\overline{\pi(X_\iota\cap(U\times V))}^{HC(V)}\,.
\]
\qed
\end{proof}

\begin{remark}
\label{rem:no-local-strong-Chevalley}
A stronger version of Chevalley's theorem~\ref{thm:chev} asserts that a projection of a $\C$-algebraic-constructible set (that is, a boolean combination of $\C$-algebraic sets) is itself a $\C$-algebraic-constructible set. One may thus expect that a local version of this result (for Nash-constructible sets) holds as well. This is not the case.

We say that a subset $X$ of an open $\Omega\subset M$ is \emph{Nash-constructible} (in $\Omega$) if for every $x_0\in\Omega$ there exist an open neighbourhood $U$ of $x_0$ in $\Omega$ and Nash subsets $X_1,\dots,X_s,Y_1\dots,Y_s$ of $U$ such that $X\cap U=\bigcup_{k=1}^s(X_k\setminus Y_k)$.

Now, let $X:=\{(z,w)\in\C^2:z+w=1\}$ and let $\Omega$ be the open polydisc $\{(z,w)\in\C^2:|z|<1,|w|<1\}$. Then the projection $\pi$ of $X\cap\Omega$ to the $z$-axis is not a Nash-constructible subset of $\pi(\Omega)$.

More generally, the main result of \cite{Mar} implies that the closure under the operations of Cartesian product,
finite union, complement and Cartesian projection of the family of sets $X$ which are semialgebraic in some $\C^n$ and 
Nash-constructible in some open $U\subset\C^n$ equals to the family of all the semialgebraic subsets of $\C^n$ (for all $n$). Therefore it is strictly larger than the family of Nash-constructible sets.
\end{remark}

\begin{theorem}
\label{thm:nash-map-image}
Let $\Omega$ be an open subset of $M$, let $\vp:\Omega\to N$ be a Nash mapping, and let $X$ be a Nash subset of $\Omega$. Then, for every point $\xi\in X$ and an arbitrarily small open neighbourhood $U$ of $\xi$, there exists an arbitrarily small open neighbourhood $V$ of $\vp(\xi)$ such that $\overline{\vp(X\cap U)\cap V}^{HC(V)}$ is a Nash subset of $V$, of dimension equal to $\dim\vp(X\cap U)$.
In particular, $\overline{(\vp(X\cap U))_{\vp(\xi)}}^{HC}$ is a Nash germ, of dimension $\dim\vp(X\cap U)$.
\end{theorem}

\begin{proof}
Let $\xi$ be a point of $X$, and let $U$ be a bounded open neighbourhood of $\xi$ in $\Omega$. Then the graph of $\vp$ restricted to $X\cap U$,
\[
\Gamma_{\vp|_{X\cap U}}=\{(x,y)\in U\times N: x\in X, y=\vp(x)\}
\]
is a Nash subset of $U\times N$, as $\Gamma_{\vp|_{X\cap U}}=(U\times N)\cap\Gamma_\vp\cap(\Omega\times N)$ is the trace on $U\times N$ of a Nash subset of $\Omega\times N$.
Let $\pi:M\times N\to N$ denote the canonical projection. Then, by Corollary~\ref{cor:nash-proj}, one can choose an arbitrarily small open neighbourhood $V$ of $\vp(\xi)$ in $N$, such that $\overline{\pi(\Gamma_{\vp|_{X\cap U}}\cap(U\times V))}^{HC(V)}$ is a Nash subset of $V$, of dimension $\dim\pi(\Gamma_{\vp|_{X\cap U}}\cap(U\times V))$. But
\[
\pi(\Gamma_{\vp|_{X\cap U}}\cap(U\times V))=\vp(X\cap U)\cap V\,,
\]
which completes the proof.
\qed
\end{proof}

\begin{example}
\label{ex:only-nash}
Note that $\overline{(\vp(X\cap U))_{\vp(\xi)}}^{HC}$ need not be a $\C$-algebraic germ, even if $\Omega=M$ and $\vp$ is a $\C$-polynomial mapping. Let, for example, $Y$ be the curve $z^2=w^2(w+1)$ in $\C^2$, let $X$ be the normalization of $Y$, and let $\vp:\C\cong X\to\C^2$ be the composite of the canonical maps $X\to Y$ and $Y\hookrightarrow\C^2$. Let $\xi\in X$ be one of the two preimages of $0\in Y$. Then, for a sufficiently small open neighbourhood $U$ of $\xi$, $\overline{(\vp(X\cap U))_0}^{HC}=(\vp(X\cap U))_0$ is the germ of one of the two $\C$-analytic branches of $Y$ near $0$, which is not a $\C$-algebraic germ, because $Y$ is irreducible.
\end{example}

\begin{remark}
\label{rem:false-in-general}
Note also that Theorem~\ref{thm:nash-map-image} is false without the assumption that $X$ is Nash. Indeed, let $\vp:\C^4\ni(w,x,y,z)\to(w,x,y)\in\C^3$ and let $X\subset\C^4$ be a $\C$-analytic set defined by equations $x=wz, y=wze^z$. Then $\dim{X}=2$, hence $\dim\vp(X\cap U)\leq2$ for an arbitrary neighbourhood $U$ of the origin in $\C^4$. However, for an arbitrarily small $U$, the germ $(\vp(X\cap U))_0$ is contained in no proper $\C$-analytic subgerm of $\C^3_0$, hence $\dim\overline{(\vp(X\cap U))_0}^{HC}\!=3>\dim\vp(X\cap U)$. This is an Osgood example (see, e.g., \cite{GR}).
\end{remark}

\section{Semialgebraic stratification by holomorphic closure dimension}
\label{sec:HC-stratification}

\subsection{Holomorphic closure of a semialgebraic set}
\label{subsec:HC-nash}

Let $S$ be a semialgebraic subset of a finite-dimensional $\C$-vector space $M$. Let $\xi$ be a point of $\bS$.
Recall that a $\C$-analytic germ $X_\xi\subset M_\xi$ is the \emph{holomorphic closure} of $S$ at $\xi$ (denoted $\overline{S_\xi}^{HC}$) if it is the smallest $\C$-analytic germ at $\xi$ containing $S_\xi$.

Proposition~\ref{prop:HC-nash} asserts that the holomorphic closure of a semialgebraic set is a Nash germ.

\subsubsection*{Proof of Proposition~\ref{prop:HC-nash}}
By \S~\ref{subsec:semialg}, $S$ can be written as a finite union of semialgebraic open subsets of $\R$-algebraic sets; i.e., sets of the form
\[
\{\zeta\in\C^n:f_1(\zeta,\bar{\zeta})=\dots=f_k(\zeta,\bar{\zeta})=0, g_1(\zeta,\bar{\zeta})>0,\dots,g_l(\zeta,\bar{\zeta})>0\}\,,
\]
where $f_1,\dots,f_k,g_1,\dots,g_l$ are polynomial functions with real coefficients.
Let $\xi\in\overline{S}$. Let $\tilde{S}^1,\dots,\tilde{S}^\mu$ be all such subset of $S$, which are adherent to $\xi$, and let $R^1,\dots,R^\mu$ be the corresponding $\R$-algebraic sets. Then, for every $j=1,\dots,\mu$, $S_\xi\cap R^j_\xi$ is an open subgerm of $R^j_\xi$.

We will now use terminology of \S~\ref{subsec:complexification}. Let $A^1,\dots,A^\mu$ be $\C$-algebraic subsets of $\C^{2n}_{(z,w)}$, such that $(R^j_\xi)^c=A^j_\eta$, where $\eta:=\dl(\xi)$. Let $R^j_\xi=R^{j,1}_\xi\cup\dots\cup R^{j,s_j}_\xi$be the decomposition of $R^j_\xi$ into irreducible $\R$-analytic germs, and let $A^j_\eta=A^{j,1}_\eta\cup\dots\cup A^{j,t_j}_\eta$ be the decomposition of $A^j_\eta$ into irreducible $\C$-analytic subgerms. Then, for every $j$, we have $t_j=s_j$ and (up to permutation of indices) $A^{j,k}_\eta$ is precisely the complexification of $R^{j,k}_\xi$ ($k=1,\dots,s_j$), by \cite[Prop.\,9]{Car}.

Now, for all $j$ and $k$, either $S_\xi\cap R^{j,k}_\xi=\varnothing$ or else $S_\xi\cap R^{j,k}_\xi$ is a non-empty open subgerm of $R^{j,k}_\xi$. Therefore the holomorphic closure $\overline{S_\xi}^{HC}$ is the union of $\overline{R^{j,k}_\xi}^{HC}$ over all pairs $(j,k)$ such that $S_\xi\cap R^{j,k}_\xi\neq\varnothing$. Consequently, it suffices to show that each $\overline{R^{j,k}_\xi}^{HC}$ is a Nash germ.

By \S~\ref{subsec:complexification}, we can identify $\overline{R^{j,k}_\xi}^{HC}$ with the smallest $\C$-analytic germ containing $(\pi^z(A^{j,k}))_{\pi^z(\eta)}$, where $A^{j,k}$ is an irreducible representative of the germ $A^{j,k}_\eta$ in some open neighbourhood $U'$ of $\eta$ in $\C^{2n}$. After shrinking $U'$ if needed, we can assume that $A^{j,k}$ is an analytic-irreducible component of $A^j\cap U'$, and hence $A^{j,k}$ is a Nash subset of $U'$ (Remark~\ref{rem:nash-facts}). Then
\[
\overline{R^{j,k}_\xi}^{HC}=\overline{(\pi^z(A^{j,k}))_{\pi^z(\eta)}}^{HC}
\]
is a Nash germ, by Theorem~\ref{thm:nash-map-image}.
\qed

\medskip

As an immediate consequence of Proposition~\ref{prop:HC-nash}, we recover the main result of \cite{FLR} (Proposition~\ref{prop:holo-semialg}, below). The global statement (``semialgebraic $\C$-analytic subset of $M$ is $\C$-algebraic'') follows from Proposition~\ref{prop:holo-semialg} via Remark~\ref{rem:nash-facts}(5) (see also \cite[Cor.\,4.5]{PS}).

\begin{proposition}
\label{prop:holo-semialg}
Let $X$ be a $\C$-analytic subset of an open set $\Omega$ in $M$. If $X$ is semialgebraic in $M$, then $X$ is a Nash subset of $\Omega$.
\end{proposition}

\begin{proof}
It suffices to show that the germ of $X$ at every point of $X$ is Nash.
Let $\xi$ be a point of $X$. By Proposition~\ref{prop:HC-nash}, the holomorphic closure $\overline{X_\xi}^{HC}$ is a Nash germ. But $X_\xi$ itself is the smallest $\C$-analytic germ containing $X_\xi$, so $X_\xi=\overline{X_\xi}^{HC}$ is a Nash germ.
\qed
\end{proof}
\medskip

In the next section, we will study complex dimensions of preimages of semialgebraic sets under holomorphic semialgebraic mappings.
It turns out that such mappings are necessarily Nash.

\begin{proposition}
\label{prop:nash-map}
Let $\Omega$ be an open semialgebraic subset of $M$. If $\vp:\Omega\to N$ is a holomorphic semialgebraic mapping, then $\vp$ is Nash.
\end{proposition}

\begin{proof}
By assumption, the graph $\Gamma_\vp$ of $\vp$ is a semialgebraic subset of $M\times N$. On the other hand, $\Gamma_\vp$ is biholomorphic with $\Omega$, and hence it is a $\C$-analytic subset of the open set $\Omega\times N$. Thus, by Proposition~\ref{prop:holo-semialg}, $\Gamma_\vp$ is a Nash subset of $\Omega\times N$; i.e., $\vp$ is a Nash mapping (Remark~\ref{rem:nash-facts}(1)).
\qed
\end{proof}

Let us note an important special case:

\begin{corollary}
\label{cor:holo-semialg-global}
If $\vp:M\to N$ is a holomorphic semialgebraic mapping of $\C$-vector spaces, then $\vp$ is a $\C$-polynomial map.
\end{corollary}

\begin{proof}
By Proposition~\ref{prop:nash-map}, the graph $\Gamma_\vp$ of $\vp$ is a Nash subset of $M\times N$. On the other hand, $\Gamma_\vp$ is an irreducible $\C$-analytic subset of $M\times N$ (as it is biholomorphic to $M$). It follows that $\Gamma_\vp$ is an irreducible $\C$-algebraic subset of $M\times N$, by Remark~\ref{rem:nash-facts}(5). Therefore $\vp$ is a holomorphic mapping with a $\C$-algebraic graph, and hence a $\C$-polynomial map, by the Serre Algebraic Graph Theorem (see, e.g., \cite[\S\,VII.16]{Loj}).
\qed
\end{proof}
\medskip

Combining Proposition~\ref{prop:nash-map} with Theorem~\ref{thm:nash-map-image}, we get the following result.

\begin{theorem}
\label{thm:holo-semialg-map}
Let $\Omega$ be a semialgebraic open subset of $M$, let $\vp:\Omega\to N$ be a holomorphic semialgebraic mapping, and let $X$ be a Nash subset of $\Omega$.
Then, for every point $\xi\in X$ and an arbitrarily small open neighbourhood $U$ of $\xi$, there exists an arbitrarily small open neighbourhood $V$ of $\vp(\xi)$ such that $\overline{\vp(X\cap U)\cap V}^{HC(V)}$ is a Nash subset of $V$, of dimension equal to $\dim\vp(X\cap U)$.
In particular, the germ $\overline{(\vp(X\cap U))_{\vp(\xi)}}^{HC}$ is Nash, of dimension $\dim\vp(X\cap U)$.
\end{theorem}

\subsection{Semialgebraic stratification by holomorphic closure dimension}
\label{subsec:HC-stratification}

\begin{lemma}
\label{lem:smooth}
Let $S$ be an $\R$-analytic submanifold of an open subset of $M$. If $X$ is a $\C$-analytic subset of an open set $\Omega\subset M$, and $X_\xi\supset S_\xi$ at some point $\xi\in\bS$, then $X$ contains the closure of a connected component of $S\cap\Omega$.
\end{lemma}

\begin{proof}
Identity Principle for $\R$-analytic functions.
\qed
\end{proof}

\begin{lemma}
\label{lem:alg-closure}
Let $S$ be a connected $\R$-analytic submanifold of an open subset of $M$. There exists a unique smallest $\C$-algebraic subset $X$ of $M$ containing $\bS$ and such that, for every $\xi\in\bS$, $X_\xi$ is the smallest $\C$-algebraic germ containing $S_\xi$. Moreover, $X$ is irreducible.
\end{lemma}

\begin{proof}
For every $\xi\in\bS$, define $X^\xi$ as the minimal (with respect to inclusion) element of the family of sets
\[
\{Z\subset M: Z\ \mathrm{is\ }\C\mathrm{-algebraic\ in\ }M,\ \mathrm{and\ }Z_\xi\supset S_\xi\}\,,
\]
which is well-defined, by Noetherianity. By Lemma~\ref{lem:smooth}, each $X^\xi$ contains $\bS$. Given any $\xi_1,\xi_2\in\bS$, we thus have $(X^{\xi_1})_{\xi_2}\supset S_{\xi_2}$, hence $X^{\xi_1}\supset X^{\xi_2}$, by minimality of $X^{\xi_2}$. Therefore $X^{\xi_1}=X^{\xi_2}$, and so the set $X:=X^\xi$ is independent of the choice of $\xi$. By construction, $X$ has the required properties.

To prove the final assertion of the lemma, suppose that $X$ is a union of two proper $\C$-algebraic subsets $X_1$ and $X_2$. Then $\bS\not\subset X_1$, $\bS\not\subset X_2$, and $\bS\not\subset X_1\cap X_2$, by minimality of $X$. But $\bS\subset X_1\cup X_2$, hence there exists a point $\xi_0\in\bS$, such that $\xi_0\in X_1\setminus X_2$, and so
$(X_1)_{\xi_0}\supset S_{\xi_0}$. Then $X_1\supset\bS$, by Lemma~\ref{lem:smooth}; a contradiction.
\qed
\end{proof}

\begin{proposition}
\label{prop:HC-closure-smooth}
Let $S$ be a connected semialgebraic subset of $M$ and an $\R$-analytic submanifold of an open subset of $M$, and let $X$ be the unique $\C$-algebraic subset of $M$ from the above lemma. Then, at every point $\xi\in\bS$, the holomorphic closure $\overline{S_\xi}^{HC}$ is a union of some analytic-irreducible components of $X_\xi$. In particular, the holomorphic closure dimension is constant on $S$.
\end{proposition}

\begin{proof}
Let $\xi$ be a point of $\bS$. By Proposition~\ref{prop:HC-nash}, $\overline{S_\xi}^{HC}=\overline{\bS_\xi}^{HC}$ is a Nash germ. Hence we can choose an open neighbourhood $U$ of $\xi$ in $M$ such that $\overline{S_\xi}^{HC}$ has a Nash representative $Y$ in $U$, satisfying $Y\supset\bS\cap U$. Let $Z$ be a $\C$-algebraic subset of $M$, for which $Y$ is a union of analytic-irreducible components of $Z\cap U$. Then $Z_\xi\supset Y_\xi\supset S_\xi$, hence $Z\supset\bS$, by Lemma~\ref{lem:smooth}. By minimality of $X$, $Z\supset X$. Hence $Z_\xi\supset X_\xi$, and consequently every analytic-irreducible component $X^j_\xi$ of $X_\xi$ is contained in an analytic-irreducible component $Z^k_\xi$ of $Z_\xi$. On the other hand, $X_\xi\supset S_\xi$, so by definition of the holomorphic closure, $X_\xi\supset Y_\xi$. Therefore, every analytic-irreducible component of $Y_\xi$ is contained in an analytic-irreducible component of $X_\xi$. As $Y_\xi$ is a union of some analytic-irreducible components of $Z_\xi$, it follows that $Y_\xi$ is a union of some components $X^j_\xi$ of $X_\xi$, as required.
Finally, the irreducibility of $X$ implies that all the analytic-irreducible components of a germ of $X$ at any point $\xi\in X$ are of dimension $\dim X$.
\qed
\end{proof}

For a semialgebraic set $S$ in $M$ of dimension $d$, denote by $\mathrm{Reg}_d(S)$ the locus of points $\xi\in S$ for which $S_\xi$ is a germ of a $d$-dimensional $\R$-analytic manifold.

\begin{corollary}
\label{cor:CR-mnfld}
Let $S$ be a $d$-dimensional semialgebraic set in $M$.
If $\mathrm{Reg}_d(S)$ is connected and dense in $S$, then there exists an irreducible $\C$-algebraic subset $X$ in $M$ such that $X\supset\bS$ and $\dim_{HC}S_\xi=\dim X$ for every $\xi\in\overline{S}$.
In particular, this is the case if $S$ is a semialgebraic CR manifold.
\end{corollary}

\begin{proof}
Let $X$ be the unique irreducible $\C$-algebraic set for $\mathrm{Reg}_d(S)$ from Lemma~\ref{lem:alg-closure}.
Then, by Proposition~\ref{prop:HC-closure-smooth}, $X$ has the required properties, because $\overline{S}=\overline{\mathrm{Reg}_d(S)}\subset\overline{X}=X$.
\qed
\end{proof}

\begin{remark}
\label{rem:CR-mnfld}
Notice that the above result does not carry over to the semianalytic case, where it is possible to have a smooth connected set with non-constant holomorphic closure dimension (see Examples 6.1 and 6.3 of \cite{AS}).
\end{remark}
\medskip

We conclude this section with the proof of the semialgebraic stratification by holomorphic closure dimension.

\subsubsection*{Proof of Theorem~\ref{thm:HC-stratification}}

Let $\xi$ be a point of $S$, and let $Y$ be a $\C$-analytic subset of a neighbourhood $U$ of $\xi$ in $M$, such that $Y_\xi=\overline{S_\xi}^{HC}$. After shrinking $U$ if needed, we can assume that $Y\cap U\supset S\cap U$, and $\dim Y_x\leq\dim Y_\xi$ for all $x\in U$. Hence $Y_x\supset\overline{S_x}^{HC}$, for every $x\in U$, and so
\[
\dim\overline{S_x}^{HC}\leq\dim Y_x\leq\dim Y_\xi=\dim\overline{S_\xi}^{HC}\,.
\]
This proves closedness of $\mathcal{S}^d(S)$, for $d\in\N$.

For the proof of semialgebraicity of $\mathcal{S}^d(S)$, consider a finite partition $\{S_\iota\}_{\iota \in I}$
of $S$ into semialgebraic sets, each of which is a connected $\R$-analytic manifold (Remark~\ref{rem:SAG-facts}(3)).
For each $\iota\in I$, there is an irreducible $\C$-algebraic set $X_\iota$ satisfying the conclusion of Lemma~\ref{lem:alg-closure} applied to $S_\iota$.
Then, for every $x\in\overline{S_\iota}$, $\dim_{HC}(S_\iota)_x=\dim X_\iota$, by Proposition~\ref{prop:HC-closure-smooth}.
For $x\in S$, let $I(x)=\{\iota \in I : x\in \overline{S_\iota} \}$.
Then $\overline{S_x}^{HC}=\bigcup_{\iota\in I(x)}\overline{(S_\iota)_x}^{HC}$, hence $\dim_{HC}S_x=\max \{ \dim X_\iota : \iota\in I(x)\}$.
Thus the holomorphic closure dimension of $S_x$ only depends on $I(x)$. But, for any $I'\subset I$, the set $\{ x\in S: I(x)=I'\}$ is semialgebraic, by Remark~\ref{rem:SAG-facts}(1).
\qed

\section{Preimages under holomorphic semialgebraic mappings}
\label{sec:preimages}

We will now study the relationship between the complex dimensions of semialgebraic sets and those of their preimages, with a view toward applications in CR geometry.

Let $\Omega$ and $\Delta$ be open connected subsets of $M$ and $N$ respectively, and let $\vp:\Omega\to\Delta$ be a holomorphic mapping.
Denote by $\lambda$ the generic fibre dimension of $\vp$, and set $\Omega^{(\lambda)}:=\{x\in\Omega:\fbd_x\vp=\lambda\}$.
Assume that $\vp$ is dominant; i.e., $\lambda=\dim M-\dim N$.

\begin{proposition}
\label{prop:HC-preimage-equidim}
Let $S$ be an arbitrary subset of $\vp(\Omega)$. Suppose that $\vp$ is equidimensional; i.e., $\fbd_x\vp=\lambda$ for all $x\in\Omega$. Then, for all $\eta\in S$ and $\xi\in\vp^{-1}(\eta)$,
\[
\dim_{HC}(\vp^{-1}(S))_\xi=\dim_{HC}S_\eta+\lambda\,.
\]
\end{proposition}

\begin{proof}
Given $\eta\in S$, we can choose an open neighbourhood $V$ of $\eta$ in $N$, such that the holomorphic closure $\overline{S_\eta}^{HC}$ has a representative $Y$ $\C$-analytic in $V$. Then $(\vp^{-1}(Y))_\xi\supset(\vp^{-1}(S))_\xi$, for every $\xi\in\vp^{-1}(\eta)$, and $\dim(\vp^{-1}(Y))_\xi=\dim{Y_\eta}+\lambda$, by Theorem~\ref{thm:rem}, as the fibre dimension of $\vp|_{\vp^{-1}(Y)}$ is constantly $\lambda$. Hence $\dim_{HC}(\vp^{-1}(S))_\xi\leq\dim_{HC}S_\eta+\lambda$, at every $\xi\in\vp^{-1}(\eta)$.

Next, suppose there exists a point $\xi\in\vp^{-1}(\eta)$ for which $\dim_{HC}(\vp^{-1}(S))_\xi<\dim_{HC}S_\eta+\lambda$. Choose an open neighbourhood $U$ of $\xi$ in $\Omega$ such that $\overline{(\vp^{-1}(S))_\xi}^{HC}$ has a representative $X$ $\C$-analytic in $U$, and such that $X\supset\vp^{-1}(S)\cap U$. Set $X^{(\lambda)}:=\{x\in X:\fbd_x\vp|_X=\lambda\}$. Then $X^{(\lambda)}$ is a $\C$-analytic subset of $U$, and $(X^{(\lambda)})_\xi\supset(\vp^{-1}(S))_\xi$, because $\vp^{-1}(S)$ consists of $\lambda$-dimensional fibres. By minimality of holomorphic closure, we have $(X^{(\lambda)})_\xi=X_\xi$, and hence we can replace $X$ with $X^{(\lambda)}$. Then the fibre dimension of $\vp$ is constant on $X$, so by Theorem~\ref{thm:rem}, $\vp(X)$ is $\C$-analytic in a neighbourhood of $\eta$, of dimension $\dim(\vp(X))_\eta=\dim X_\xi+\lambda$.

By assumption, $S=\vp(\vp^{-1}(S))$, and hence $\vp(X)\supset\vp(\vp^{-1}(S)\cap U)=S\cap\vp(U)$. Note that $\vp(U)$ is an open neighbourhood of $\eta$, because $\vp$ is an open mapping, by the Remmert Open Mapping Theorem (see, e.g., \cite{Loj}). Hence $(\vp(X))_\eta\supset S_\eta$, and so
\[
\dim_{HC}S_\eta\leq\dim(\vp(X))_\eta=\dim X_\xi-\lambda<(\dim_{HC}S_\eta+\lambda)-\lambda\,;
\]
a contradiction.
\qed
\end{proof}
\medskip

Let, as before, $\mathcal{S}^d(S)$ denote the set of points $x\in S$ such that $\dim_{HC}S_x\geq d$.

\begin{corollary}
\label{cor:HC-preimage-equidim}
Under the hypotheses of Proposition~\ref{prop:HC-preimage-equidim},
\[
\vp^{-1}(\mathcal{S}^d(S))=\mathcal{S}^{d+\lambda}(\vp^{-1}(S))\,,
\]
for all $d\in\N$.
\end{corollary}

Without the equidimensionality assumption on $\vp$, we have the following:

\begin{proposition}
\label{prop:HC-preimage-ineq}
Let $S$ be an arbitrary subset of $\vp(\Omega)$. Suppose that for every $\eta\in S$, $S_\eta\cap(\vp(\Omega\setminus\Omega^{(\lambda)}))_\eta$ is nowhere-dense in $S_\eta$. Then, for all $\eta\in S$ and $\xi\in\vp^{-1}(\eta)$, we have
\begin{equation}
\label{eq:prop}
\dim_{HC}(\vp^{-1}(S))_\xi\leq\dim_{HC}S_\eta+\lambda\,,
\end{equation}
provided $(\vp^{-1}(S))_\xi\cap(\Omega\setminus\Omega^{(\lambda)})_\xi$ is nowhere-dense in $(\vp^{-1}(S))_\xi$.
\end{proposition}

\begin{proof}
Given $\eta\in S$, choose an open neighbourhood $V$ of $\eta$ in $N$ such that $\overline{S_\eta}^{HC}$ has a representative $Y$ $\C$-analytic in $V$.
Then $\vp^{-1}(Y)$ is a $\C$-analytic subset of $\vp^{-1}(V)$, and it contains $\vp^{-1}(S)\cap\vp^{-1}(V)$.
Let $\xi\in\vp^{-1}(\eta)$ be such that $(\vp^{-1}(S))_\xi\cap(\Omega\setminus\Omega^{(\lambda)})_\xi$ is nowhere-dense in $(\vp^{-1}(S))_\xi$.
Let $U$ be an open neighbourhood of $\xi$ in $\Omega$, such that $\vp^{-1}(Y)\cap U$ has a finite number of irreducible components; say, $\Sigma^1,\dots,\Sigma^s$. Then, by assumption, $(\vp^{-1}(S))_\xi$ is contained in the union of those $\Sigma^j_\xi$ for which $\Sigma^j\not\subset(\Omega\setminus\Omega^{(\lambda)})\cap U$. Hence, replacing $\vp^{-1}(Y)\cap U$ by the union of these $\Sigma^j$, if needed, we can assume that there exists a component $\Sigma^j$ of $\vp^{-1}(Y)\cap U$, of dimension $\dim(\vp^{-1}(Y))_\xi$, and such that the generic fibre dimension of $\vp|_{\Sigma^j}$ is $\lambda$. Therefore, by Theorem~\ref{thm:rem},
\[
\dim(\vp^{-1}(Y))_\xi=\dim\Sigma^j=\dim(\vp(\Sigma^j))_\eta+\lambda\leq\dim Y_\eta+\lambda\,,
\]
and hence $\dim_{HC}(\vp^{-1}(S))_\xi\leq\dim Y_\eta+\lambda$, as required.
\qed
\end{proof}

In the semialgebraic setting we can get even more: Assume that $\Omega$ and $\Delta$ are semialgebraic in $M$ and $N$ respectively, and that $\vp:\Omega\to\Delta$ is holomorphic semialgebraic.

\begin{proposition}
\label{prop:HC-preimage}
Under the hypotheses of Proposition~\ref{prop:HC-preimage-ineq}, assume furthermore that $S$ is a semialgebraic subset of $N$ and $\vp$ is a semialgebraic mapping.
If $S$ is of constant holomorphic closure dimension, then we have equality in \eqref{eq:prop}.
\end{proposition}

\begin{proof}
Let $h$ denote the constant holomorphic closure dimension of $S$.
Choose a point $\xi\in\vp^{-1}(S)\cap\Omega^{(\lambda)}$, and suppose that $\dim_{HC}(\vp^{-1}(S))_\xi<h+\lambda$.
By Proposition~\ref{prop:HC-nash}, the holomorphic closure $\overline{(\vp^{-1}(S))_\xi}^{HC}$ is a Nash germ. Therefore, we can choose an open neighbourhood $U$ of $\xi$ in $\Omega$ such that $\overline{(\vp^{-1}(S))_\xi}^{HC}$ has a Nash representative $X$ in $U$, with $\dim X_\xi=\dim_{HC}(\vp^{-1}(S))_\xi$, and $X\supset\vp^{-1}(S)\cap U$. After shrinking $U$ if needed, we can also assume that $U\subset\Omega^{(\lambda)}$. Then $\vp|_U$ is an open mapping, by Remmert's Open Mapping Theorem, and hence $\vp(U)$ is an open neighbourhood of $\eta:=\vp(\xi)$.
Let $X^{(\lambda)}:=\{x\in X:\fbd_x(\vp|_X)\geq\lambda\}$. Then $X^{(\lambda)}\supset\vp^{-1}(S)\cap U$, and $X^{(\lambda)}$ is a Nash subset of $U$, so by minimality of $X_\xi$, we can replace $X$ with $X^{(\lambda)}$. After shrinking $U$ if necessary, we can assume that $X$ has a finite number of irreducible components; say, $\Sigma^1,\dots,\Sigma^s$. For each $j=1,\dots,s$, let $\lambda_j$ denote the generic fibre dimension of $\vp|_{\Sigma^j}$. Then we have
\[
\dim\vp(\Sigma^j)=\dim\Sigma^j-\lambda_j\leq\dim X_\xi-\lambda\,,
\]
and hence $\dim\vp(X)\leq\dim X_\xi-\lambda$, since $\vp(X)=\bigcup_j\vp(\Sigma^j)$.
By Theorem~\ref{thm:holo-semialg-map}, $\overline{(\vp(X))_\eta}^{HC}$ is a Nash germ at $\eta$ of dimension $\dim\vp(X)$.
On the other hand,
\[
\vp(X)\supset\vp(\vp^{-1}(S)\cap U)=S\cap\vp(U)\,,
\]
hence, by openness of $\vp(U)$ in $N$, $\overline{(\vp(X))_\eta}^{HC}$ contains the holomorphic closure $\overline{S_\eta}^{HC}$. Therefore
\begin{multline}
\notag
h=\dim\overline{S_\eta}^{HC}\leq\dim\overline{(\vp(X))_\eta}^{HC}=\dim\vp(X)\leq\dim X_\xi-\lambda\\
=\dim_{HC}(\vp^{-1}(S))_\xi-\lambda<(h+\lambda)-\lambda\,;
\end{multline}
a contradiction. We have thus proved that $\dim_{HC}(\vp^{-1}(S))_\xi=h+\lambda$ for all $\xi\in\vp^{-1}(S)\cap\Omega^{(\lambda)}$.

To complete the proof, consider a point $\xi\in\vp^{-1}(S)\setminus\Omega^{(\lambda)}$. By assumption, $(\vp^{-1}(S))_\xi\cap(\Omega\setminus\Omega^{(\lambda)})_\xi$ is nowhere-dense in $(\vp^{-1}(S))_\xi$, hence arbitrarily close to $\xi$ there are points $x$ of $\vp^{-1}(S)$ such that $\dim_{HC}(\vp^{-1}(S))_x=h+\lambda$. Hence $\dim_{HC}(\vp^{-1}(S))_\xi\geq h+\lambda$, by upper semicontinuity of the holomorphic closure dimension. The opposite inequality was shown in Proposition~\ref{prop:HC-preimage-ineq}.
\qed
\end{proof}
\medskip

The proof of Proposition~\ref{prop:HC-preimage} is of local nature. Therefore we can derive from it the following result (under the assumption that $\vp$ is dominant).

\begin{proposition}
\label{prop:HC-preimage-bis}
Let $S$ be a semialgebraic subset of $N$. Then, for all $\eta\in S$ and $\xi\in\vp^{-1}(\eta)\cap\Omega^{(\lambda)}$, we have
\[
\dim_{HC}(\vp^{-1}(S))_\xi=\dim_{HC}S_\eta+\lambda\,.
\]
\end{proposition}

\begin{proof}
One can repeat the (relevant part of the) proof of Proposition~\ref{prop:HC-preimage} verbatim, because for a point $\xi\in\Omega^{(\lambda)}$ one can choose an open neighbourhood $U$ such that $U\subset\Omega^{(\lambda)}$, by upper semicontinuity of fibre dimension.
\qed
\end{proof}

\section{Applications to CR geometry}
\label{sec:applications}

\begin{proposition}
\label{prop:semialg-smooth-CR-mnfld}
Let $S$ be a $d$-dimensional semialgebraic set in $M$, of constant holomorphic closure dimension $h$.
Then $S$ contains a closed semialgebraic set $T$ of dimension less than $d$, such that $S\setminus T$ is a CR manifold of CR dimension $d-h$.
\end{proposition}

\begin{proof}
The proof is an easy adaptation of that of \cite[Thm.\,1.5]{AS}; we include it for the reader's convenience.

Let $\mathrm{Gr}(2n,d)$ denote the space of $d$-dimensional $\R$-linear subspaces of $\C^n\cong\R^{2n}$, and let $G_k$ be the subset of $\mathrm{Gr}(2n,d)$ consisting of $d$-dimensional subspaces of $\C^n$ of CR dimension at least $k$. Let $L\in\mathrm{Gr}(2n,d)$ be a subspace generated by vectors $v^1,\dots,v^d$, $v^j=(v^j_1,\dots,v^j_{2n})\in\R^{2n}$, $j=1,\dots,d$, and let $w^j=(w^j_1,\dots,w^j_{2n}):=Jv^j$ (where $Jv$ means the product $\sqrt{-1}\cdot v$, with $v$ viewed as a vector in $\C^n$). Then $L\in G_k$ if and only if in any matrix of the form
\begin{equation*}
\begin{bmatrix}
v^1_1 & \cdots & v^d_1 & w^{1}_1 & \cdots & w^{d-2k+1}_1 \\
\vdots & & \vdots & \vdots & & \vdots\\
v^1_{2n} & \cdots & v^d_{2n} & w^{1}_{2n} & \cdots & w^{d-2k+1}_{2n}
\end{bmatrix}
\end{equation*}
every $(2d-2k+1)$-minor vanishes, provided $2d-2k+1\leq2n$.

Let $S^\iota$ be a connected component of the $\R$-analytic manifold $\mathrm{Reg}_d(S)$. The assignment $p\mapsto T_pS^\iota$ defines a real-analytic semialgebraic map from $S^\iota$ to $\mathrm{Gr}(2n,d)$. Hence, for every $0\leq k\leq\left[\frac{d}{2}\right]$, the set
\[
S^\iota_k:=\{x\in S^\iota: \dim_{CR}T_xS^\iota\geq k\}
\]
is an $\R$-analytic semialgebraic subset of $S^\iota$. If $m_\iota=\min\{\dim_{CR}T_pS^\iota: p\in S^\iota\}$, then $S^\iota_{m_\iota+1}$ is a proper analytic subset of $S^\iota$ (hence of dimension less than $d$) and $S^\iota\setminus S^\iota_{m_\iota+1}=S^\iota_{m_\iota}\setminus S^\iota_{m_\iota+1}$ is a CR-manifold of CR dimension $m_\iota$.
By \cite[Prop.\,1.4]{AS}, $m_\iota=d-h$ for all $\iota$.
The result thus follows by setting $T$ to be the union of all the $S^\iota_{m_\iota+1}$ and the set
$S\setminus\mathrm{Reg}_d(S)$.
\qed
\end{proof}

\begin{corollary}
\label{cor:semialg-CR-stratification}
Every semialgebraic subset $S$ of $M$ is a disjoint union of a finite family of CR manifolds which are semialgebraic in $M$.
\end{corollary}

\begin{proof}
We proceed by induction on $d:=\dim_{\R} S$.
If $d=0$, there is nothing to show. Suppose then that $d\geq1$.

As a semialgebraic set, $S$ is a disjoint union of a finite family of connected $\R$-analytic manifolds which are semialgebraic in $M$ 
(Remark~\ref{rem:SAG-facts}(3)). Thus, without loss of generality, we can assume that $S$ is connected and $S=\mathrm{Reg}_d(S)$.
By Corollary~\ref{cor:CR-mnfld}, $S$ has constant holomorphic closure dimension.
Let $T$ be as in Proposition~\ref{prop:semialg-smooth-CR-mnfld}. Since $T$ is nowhere dense in $S$, $\dim_{\R} T < \dim_{\R} S$. By the inductive hypothesis, $T$ is partitioned into a finite family $\{T_\iota\}_{\iota\in I}$ of semialgebraic CR manifolds. The finite family $\{T_\iota\}_{\iota\in I}\cup \{ S\setminus T\}$ therefore realizes a partition of $S$ into CR manifolds which are semialgebraic in $M$.
\qed
\end{proof}

Theorem~\ref{thm:semialg-CR-stratification} now follows immediately:

\subsubsection*{Proof of Theorem~\ref{thm:semialg-CR-stratification}}
To get a semialgebraic partition of $S$ by CR manifolds compatible with the family $\{\mathcal{S}^d(S)\}_{d\in\N}$, apply Corollary~\ref{cor:semialg-CR-stratification} to the semialgebraic sets $\mathcal{S}^d(S)\setminus\mathcal{S}^{d+1}(S)$, $d\in\N$.

By further stratifying, if needed, we can assume that the strata $S_\iota$ satisfy the ordinary condition of the frontier; i.e., that $S_j\cap\overline{S_k}=\varnothing$ or else $S_j\subset\overline{S_k}$ and $\dim_\R{S_j}<\dim_\R{S_k}$, for any $j,k$ (cf. \cite[Prop.\,9.1.8]{BCR}). 
Let us then choose $j,k\in I$ such that $S_j\subset\overline{S_k}$. By Corollary~\ref{cor:CR-mnfld}, $S_j$ and $S_k$ have constant holomorphic closure dimensions, say $h_j$ and $h_k$ resp., and there is an irreducible $\C$-algebraic set $X_k$ in $M$ such that $X_k\supset\overline{S_k}$ and $h_k=\dim X_k$. Then $S_j\subset\overline{S_k}\subset X_k$, and hence $h_j\leq\dim{X_k}=h_k$, as required.
\qed

\begin{remark}
\label{rem:no-analytic-CR-tameness}
It is interesting to consider the above results in the semianalytic setting. By \cite[Thm.\,1.5]{AS}, an irreducible $\R$-analytic set of pure dimension is a CR manifold outside a nowhere-dense semianalytic subset. Therefore the inductive argument of Corollary~\ref{cor:semialg-CR-stratification} carries over to the case when $S$ is semianalytic. On the other hand, a semianalytic stratification \emph{compatible with the family $\{\mathcal{S}^d(S)\}_{d\in\N}$} does not exist in general for $S$ semianalytic. Indeed, every CR manifold is of constant holomorphic closure dimension, by \cite[Prop.\,1.4]{AS}, but this dimension is not tame on semianalytic sets (Remark~\ref{rem:no-analytic-HC-tameness}).
\end{remark}
\medskip

The following two results are concerned with the relationship between CR structures of a semialgebraic set and of its preimage under a holomorphic semialgebraic mapping.
Here, as before, $\Omega^{(\lambda)}$ denotes the set of those $x\in\Omega$ for which $\fbd_x\vp=\lambda$.

\begin{theorem}
\label{thm:CR-back-and-forth}
Let $\Omega$ and $\Delta$ be open semialgebraic subsets of $M$ and $N$ respectively, and let $\vp:\Omega\to\Delta$ be a dominant holomorphic semialgebraic mapping, with generic fibre dimension $\lambda$. Let $S$ be a semialgebraic subset of $N$ such that $\dim(S\cap\vp(\Omega\setminus\Omega^{(\lambda)})<\dim{S}$ and $\dim(\vp^{-1}(S)\setminus\Omega^{(\lambda)})<\dim\vp^{-1}(S)$.
\begin{itemize}
\item[(i)] If $S$ is a CR manifold of CR dimension $m$, then there is a closed semialgebraic set $T'$ in $M$, of dimension $\dim{T'}<\dim\vp^{-1}(S)$, and such that $\vp^{-1}(S)\setminus T'$ is a CR manifold of CR dimension
\[
\dim_{CR}\vp^{-1}(S)=\dim\vp^{-1}(S)-\dim{S}+m-\lambda\,.
\]
\item[(ii)] If $\vp^{-1}(S)$ is a CR manifold of CR dimension $m'$, then there is a closed semialgebraic set $T$ in $N$, of dimension $\dim{T}<\dim{S}$, and such that $S\setminus T$ is a CR manifold of CR dimension
\[
\dim_{CR}S=\dim{S}-\dim\vp^{-1}(S)+m'+\lambda\,.
\]
\end{itemize}
\end{theorem}

\begin{proof}
Suppose first that $S$ is a CR manifold of CR dimension $m$. Then, by \cite[Prop.\,1.4]{AS}, $S$ has a constant holomorphic closure dimension $h=\dim{S}-m$. Therefore $\vp^{-1}(S)\cap\Omega^{(\lambda)}$, which is a semialgebraic set of dimension equal to $\dim\vp^{-1}(S)$, has constant holomorphic closure dimension $h+\lambda$, by Proposition~\ref{prop:HC-preimage-bis}. Hence, by Proposition~\ref{prop:semialg-smooth-CR-mnfld}, $\vp^{-1}(S)\cap\Omega^{(\lambda)}$ contains a closed semialgebraic subset $\widetilde{T}$, of dimension $\dim\widetilde{T}<\dim\vp^{-1}(S)$, such that $(\vp^{-1}(S)\cap\Omega^{(\lambda)})\setminus\widetilde{T}$ is a CR manifold, of CR dimension
\[
\dim_{CR}(\vp^{-1}(S)\cap\Omega^{(\lambda)})\setminus\widetilde{T}=\dim\vp^{-1}(S)-(h+\lambda)=\dim\vp^{-1}(S)-\dim{S}+m-\lambda\,.
\]
Now, the set $T':=\widetilde{T}\cup(\vp^{-1}(S)\setminus\Omega^{(\lambda)})$ has the required properties.

Next, suppose that $\vp^{-1}(S)$ is a CR manifold, of CR dimension $m'$. Then, by \cite[Prop.\,1.4]{AS} again, $\vp^{-1}(S)$ is of constant holomorphic closure dimension $h'=\dim\vp^{-1}(S)-m'$. Therefore $S\setminus\vp(\Omega\setminus\Omega^{(\lambda)})$ is a semialgebraic set of dimension equal to $\dim{S}$, and of constant holomorphic closure dimension $h'-\lambda$, by Proposition~\ref{prop:HC-preimage-bis}. Hence, by Proposition~\ref{prop:semialg-smooth-CR-mnfld}, $S\setminus\vp(\Omega\setminus\Omega^{(\lambda)})$ contains a closed semialgebraic subset $\widehat{T}$, of dimension $\dim\widehat{T}<\dim{S}$, such that $(S\setminus\vp(\Omega\setminus\Omega^{(\lambda)}))\setminus\widehat{T}$ is a CR manifold, of CR dimension
\[
\dim_{CR}(S\setminus\vp(\Omega\setminus\Omega^{(\lambda)}))\setminus\widehat{T}=\dim{S}-(h'-\lambda)=\dim{S}-\dim\vp^{-1}(S)+m'+\lambda\,.
\]
Then the set $T:=\widehat{T}\cup(S\cap\vp(\Omega\setminus\Omega^{(\lambda)}))$ has the required properties.
\qed
\end{proof}

We will say that a semialgebraic set $S$ in $M$ admits a \emph{holomorphic semialgebraic desingularization}, if there exists a $\C$-vector space $M'$, a holomorphic semialgebraic generically finite mapping $\sigma:M'\to M$, and a closed nowhere-dense semialgebraic subset $\Sigma$ of $M$, all such that $\sigma|_{M'\setminus\sigma^{-1}(\Sigma)}:M'\setminus\sigma^{-1}(\Sigma)\to M\setminus\Sigma$ is a biholomorphism, $S\cap\Sigma$ is nowhere-dense in $S$, the \emph{strict transform} $S':=\overline{\sigma^{-1}(S\setminus\Sigma)}$ of $S$ is an $\R$-analytic manifold, and $\sigma|_{S'}:S'\to S$ is a proper surjection.

\begin{proposition}
\label{prop:semialg-CR-desing}
Let $S$ be a semialgebraic set in $M$, which admits a holomorphic semialgebraic desingularization (with centre $\Sigma$). Suppose that the strict transform $S'$ of $S$ has constant holomorphic closure dimension (which is the case, e.g., when $S'$ is connected); say, $h$. Then $\dim_{HC}S_\xi=h$ for all $\xi\in S\setminus\Sigma$. Moreover, $S$ has a closed nowhere-dense semialgebraic subset $T$, such that $S\setminus T$ is a CR manifold of CR dimension $\dim{S}-h$.
\end{proposition}

\begin{proof}
Let $\sigma:M'\to M$ be as above, and let $S'\subset M'$ denote the strict transform of $S$. Let $h$ denote the constant holomorphic closure dimension of $S'$. Then, by Proposition~\ref{prop:HC-preimage-bis}, $\dim_{HC}S_\eta=h$ for all $\eta\in S\setminus\Sigma$, since the generic fibre dimension of $\sigma$ is zero.

By Proposition~\ref{prop:semialg-smooth-CR-mnfld}, there is a semialgebraic set $T'$ closed and nowhere-dense in $S'$, such that $S'\setminus T'$ is a CR manifold of CR dimension $\dim{S'}-h$. Then $\sigma(S'\setminus(T'\cup\sigma^{-1}(\Sigma)))$ is a CR manifold of CR dimension $\dim{S'}-h=\dim{S}-h$, since a biholomorphism preserves the CR structure and dimensions. Moreover, $\sigma(T'\setminus\sigma^{-1}(\Sigma))$ is semialgebraic and nowhere-dense in $S$, and $\sigma(T')$ is a closed subset of $S$, by properness of $\sigma$. Hence $T:=\sigma(T')\cup\Sigma$ has the required properties.
\qed
\end{proof}
\medskip

\begin{example}
Let $R$ be an $\R$-algebraic subset of $\C^3$ given by equations
\[
x_3^3-x_1^2x_2x_3=x_1^4,\ \ y_1=0,
\]
where $z_j=x_j+iy_j$, $j=1,2,3$. Then $R$ has a holomorphic semialgebraic desingularization by a single blow-up of $\C^3$ with centre the $z_2$-axis. Indeed, under the mapping $\sigma:\C^3\to\C^3$ given by $z_1=u_1, z_2=u_2$, and $z_3=u_1u_3$, where $u_j=v_j+iw_j$ ($j\leq3$), $R$ has the strict transform
\[
R'=\{(u_1,u_2,u_3)\in\C^3:v_3^3-v_2v_3=v_1, w_1=0\}\,.
\]
Now, $R'$ is an $\R$-analytic submanifold of $\C^3$ of real dimension 4. By Corollary~\ref{cor:CR-mnfld}, $R'$ has constant holomorphic closure dimension. It is easy to see that $R'_0$ is contained in no germ of a $\C$-analytic hypersurface, and hence $\dim_{HC}R'_\xi=3$ for all $\xi\in R'$. Hence $R$ is generically (i.e., outside a nowhere-dense semialgebraic set) a CR manifold of CR dimension 1, by Proposition~\ref{prop:semialg-CR-desing}.
\end{example}


\begin{acknowledgements}
The authors would like to thank Rasul Shafikov for his many useful remarks and suggestions regarding the manuscript.
\end{acknowledgements}


\end{document}